\documentclass[12pt]{amsart}
\usepackage[all]{xy}
\usepackage{amssymb}
\usepackage{amsthm}
\usepackage{hyperref}
\hypersetup{colorlinks=true,linkcolor=blue,citecolor=magenta}
\usepackage{amsmath}
\usepackage{amscd,enumitem}
\usepackage{verbatim}
\usepackage{eurosym}
\usepackage{graphicx}
\usepackage{float}
\usepackage{color}
\usepackage{dcolumn}
\usepackage[mathscr]{eucal}
\usepackage[all]{xy}
\usepackage{hyperref}
\usepackage{bbm}
\usepackage[textheight=8.65in, textwidth=6.8in]{geometry}
\usepackage{multirow}
\usepackage{caption}
\newtheorem*{thm*}{Theorem}
\newtheorem*{conj*}{Conjecture}

\newtheorem*{remark}{Remark}

\newtheorem{thm}{Theorem}[section]
\newtheorem{lem}{Lemma}[section]

\newtheorem{prop}[thm]{Proposition}

\newcommand{\Z}{\mathbb{Z}}
\newcommand{\Q}{\mathbb{Q}}

\newcommand{\SL}{\operatorname{SL}}

\newcommand{\leg}[2]{\genfrac{(}{)}{}{}{#1}{#2}}

\numberwithin{equation}{section}

\begin{document}
\title[Variations of Lehmer's Conjecture]{Variations of Lehmer's Conjecture for Ramanujan's tau-function}
\author{Jennifer S. Balakrishnan, William Craig and Ken Ono}
\address{Department of Mathematics and Statistics, Boston University,
Boston, MA 02215}
\email{jbala@bu.edu}
\address{Department of Mathematics, University of Virginia, Charlottesville, VA 22904}
\email{wlc3vf@virginia.edu}
\email{ken.ono691@virginia.edu}
\thanks{The first author acknowledges the support of the NSF (DMS-1702196), the Clare Boothe Luce Professorship
(Henry Luce Foundation), a Simons Foundation grant (Grant \#550023), and a Sloan
Research Fellowship. The third author thanks the support of the Thomas Jefferson Fund and the NSF
(DMS-1601306).}
\keywords{Lehmer's Conjecture}

\begin{abstract} 
We consider natural variants of Lehmer's unresolved conjecture that Ramanujan's tau-function never vanishes.
Namely, for $n>1$ we prove that
$$\tau(n)\not \in \{\pm 1, \pm 3, \pm 5, \pm 7, \pm 691\}.
$$
This result is an example of general theorems (see Theorems 1.2 and 1.3 of \cite{BCO}) for newforms with trivial mod 2 residual Galois representation. Ramanujan's well-known congruences for $\tau(n)$ allow for the simplified proof in these special cases.  We make use of the theory of Lucas sequences, the Chabauty--Coleman method for hyperelliptic curves, and facts
about certain Thue equations.
\end{abstract}
\maketitle
\section{Introduction and statement of results}

In his famous paper ``On certain arithmetical functions,'' Ramanujan introduced $\tau(n)$, the  Fourier coefficients of 
 (note: $q:=e^{2\pi i z}$ throughout)
\begin{equation}
\Delta(z)=\sum_{n=1}^{\infty}\tau(n)q^n:=q\prod_{n=1}^{\infty}(1-q^n)^{24}=q-24q^2+252q^3-1472q^4+4830q^5-\cdots,
\end{equation}
the normalized weight 12 cusp form for $\SL_2(\Z)$.
The tau-function has been a remarkable testing ground for the theory of modular forms.
Its multiplicative properties foreshadowed the theory of Hecke operators. Ramanujan conjectured bounds that are now celebrated corollaries of
Deligne's proof of the Weil Conjectures. Furthermore, Serre \cite{SerreRamanujan}  viewed its exceptional congruences \cite{RamanujanUnpublished, Ramanujan}
\begin{equation}\label{RamanujanCongruences}
\tau(n)\equiv \begin{cases} n^2\sigma_1(n)\!\!\!\!\!\pmod 9,\\
 n\sigma_1(n)\!\!\!\pmod 5,\\
n\sigma_3(n)\!\!\!\pmod 7,\\
\sigma_{11}(n)\!\!\pmod{691},
\end{cases}
\end{equation}
where $\sigma_\nu(n):=\sum_{d\mid n}d^{\nu},$
 as hints of a theory of modular $\ell$-adic Galois representations, which are now ubiquitous in number theory.

Surprisingly, Lehmer's Conjecture \cite{Lehmer} that $\tau(n)$ never vanishes remains open.\footnote{Recent work by Calegari and Sardari \cite{CalegariSardari} considers a different aspect. They prove that
 at most finitely many non-CM newforms with fixed tame $p$ level $N$ have vanishing $p$th Fourier coefficient.}
 We investigate a variation of the original speculation that has been previously considered. For odd $\alpha$,
 Murty, Murty and Saradha \cite{MMS} proved that
 $\tau(n)\neq \alpha$ for sufficiently large $n$. 
 Due to the gigantic bounds that  arise when applying the theory of linear forms in logarithms, which is the main technique of their proof, the classification
 of such $n$ has not been carried out for any $\alpha\neq \pm 1$. 
 We prove the following theorem.
 
\begin{thm}\label{Lehmer135}
If  $n>1,$ then we have that
$$
\tau(n) \not \in \{\pm 1, \pm 3, \pm 5, \pm 7, \pm 691\}.
$$
\end{thm}

\begin{remark} The authors and Tsai have obtained more general  (and stronger) results  \cite{BCO} for newforms with trivial mod 2 residual Galois representations.
For $\tau(n)$ with $n>1$, we have proved (see Theorem 1.2 of \cite{BCO})  that  $$\tau(n)\not \in \{\pm 1, \pm 3, \pm 5, \pm 7, \pm 13, \pm 17, -19, \pm 23,  \pm 37, \pm 691\}.$$
Assuming GRH, we also show that $$\tau(n)\not \in 
\left \{ \pm  \ell\ : \ 41\leq  \ell\leq 97  \  {\text {\rm with}}\ \leg{\ell}{5}=-1\right\} \cup
\left \{-11, -29, -31, -41, -59, -61, -71, -79, -89\right\}.
$$
The proof of Theorem~\ref{Lehmer135} here is simplified by the knowledge of Ramanujan's congruences (\ref{RamanujanCongruences}).
\end{remark}

The proof of Theorem~\ref{Lehmer135} makes use of a number of important tools in concert with (\ref{RamanujanCongruences}).
The deep work of Bilu, Hanrot, and Voutier \cite{BHV} on primitive prime divisors of Lucas sequences forms the primary framework for the proof. Suppose that $\ell \in \{3, 5, 7, 691\}$ and that $\tau(n)=\pm \ell$.
Their theory, combined with (\ref{RamanujanCongruences}) and the multiplicativity of $\tau(n)$, implies that $n=p^{d-1}$, where $p$  is an odd prime, and $d\mid \ell (\ell^2-1)$ are certain odd primes. For $\ell \in \{3, 5, 7\}$, it turns out that one must have $d=\ell$.
The condition that
$$
\tau(p^{d-1})=\pm \ell
$$
implies the existence of a specific integer point on one of two algebraic curves determined by $d$. These 
curves are of hyperelliptic and Thue-type. The proof of Theorem~\ref{Lehmer135}
follows from the explicit determination of the integer points on these curves. This classification is achieved using the Chabauty--Coleman method \cite{coleman}
and the Bilu--Hanrot algorithm \cite{BH96} for solving Thue equations.

\section*{Acknowledgements} 
\noindent
The authors thank the referees, Matthew Bisatt, Michael Griffin, Guillaume Hanrot, Sachi Hashimoto, C\'eline Maistret, Drew Sutherland, Wei-Lun Tsai, and Charlotte Ure for their helpful comments during the preparation of this paper and \cite{BCO}. The authors are particularly grateful to Guillaume Hanrot and Wei-Lun Tsai who
offered assistance with various computer calculations.
The third author thanks the Centre di Ricerca Matematica Ennio De Giorgi (Pisa, Italy) and the organizers of the 2018 conference on modular forms and Drinfeld modules for their kind hospitality.

\section{Nuts and bolts}

The proof of Theorem~\ref{Lehmer135} requires facts about primitive prime divisors of Lucas sequences,
the Hecke multiplicative properties of $\tau(n)$, and certain arithmetic facts about specific hyperelliptic curves and Thue equations.
We record these facts in this section.

\subsection{Lucas sequences and their prime divisors}
We recall the important work of Bilu, Hanrot, and Voutier \cite{BHV} on Lucas sequences.
Suppose that $\alpha$ and $\beta$ are algebraic integers for which $\alpha+\beta$ and $\alpha \beta$
are relatively prime integers, where $\alpha/\beta$ is not a root of unity.
These algebraic integers generate a {\it Lucas sequence} $\{u_n(\alpha,\beta)\}=\{u_1=1, u_2=\alpha+\beta,\dots\},$  the integers
\begin{equation}
u_n(\alpha,\beta):=\frac{\alpha^n-\beta^n}{\alpha-\beta}.
\end{equation}

A prime  $\ell \mid u_{n}(\alpha,\beta)$ is a {\it primitive prime divisor of $u_n(\alpha,\beta)$} if $\ell \nmid (\alpha-\beta)^2 u_1(\alpha,\beta)\cdots u_{n-1}(\alpha, \beta)$.  
Those $u_n(\alpha,\beta),$ where $n>2$, without a primitive prime divisor are called {\it defective\footnote{We do not consider the absence of
a primitive prime divisor for $u_2(\alpha,\beta)=\alpha+\beta$ to be   a defect.} .}
In the most famous Lucas sequence, the Fibonacci numbers, the following underlined terms are defective:
$$
1, 1, 2, 3, 5, \underbar{8}, 13, 21, 34, 55, 89, \underbar{144}, 233, 377,\dots
$$
 (note. The terms  $F_1=1$ and $F_2=1$ are not defective as their indices do not exceed $2.$)
In 1913 Carmichael \cite{Carmichael} proved that 144 is the largest defective Fibonacci number.
Bilu, Hanrot, and Voutier \cite{BHV} proved the definitive result for all Lucas sequences.
They proved that every Lucas number $u_n(\alpha,\beta)$, with $n>30,$
has a primitive prime divisor.
Their work is even more impressive; it is sharp and comprehensive.
There are sequences for which $u_{30}(\alpha,\beta)$
is defective.  Their work, combined with a subsequent paper\footnote{This paper includes a few cases which were omitted in Tables 3 and 4 of \cite{BHV}.}
by Abouzaid \cite{Abouzaid},  gives the {\it complete classification} of
defective Lucas numbers.
Tables 1-4 in Section 1 of \cite{BHV} and Theorem 4.1 of \cite{Abouzaid} offer this
classification. Every defective Lucas number either belongs to a  finite list of sporadic examples, or
a finite list of parameterized infinite families.

To study $\tau(n)$, we make use of the following consequence of their classification.

\begin{lem}\label{Modularity}
Suppose that $\alpha$ and $\beta$ are roots of the monic quadratic integral polynomial
$$
F(X)=X^2-AX+p^{11}=(X-\alpha)(X-\beta),
$$
where $p$ is an odd prime, $|A|=|\alpha+\beta|\leq 2p^{\frac{11}{2}},$ and $\gcd(\alpha+\beta, p)=1.$ 
Then there are no defective Lucas numbers $\{u_n(\alpha,\beta)\}\in \{\pm 1, \pm \ell\},$ where $\ell$ is prime.
\end{lem}
\begin{proof}
The proof uses 
Tables 1-4 of \cite{BHV}  and Theorem 4.1 of \cite{Abouzaid}.  Using the assumption that $\alpha \beta$ is the $11$th power of an odd prime, one finds that
these
Lucas numbers are not among the sporadic defective examples. 

A straightforward case-by-case analysis of 
the parameterized infinite families,
using elementary congruences and the truth of Catalan's conjecture \cite{Catalan},
that $2^3$ and $3^2$ are the only consecutive perfect powers, leaves one type of possibility.
If $|u_n(\alpha,\beta)|=\ell$, where $\ell$ is prime, then $n=\ell=3$, and $\alpha+\beta=\pm m$, where
$(p,\pm m)$ is an integer point on one of the hyperelliptic curves
\begin{equation}\label{hyper1}
Y^2 =X^{11}+3 \ \ \ {\text {\rm or}}\ \ \ Y^2=X^{11}-3.
\end{equation}
The integer points on these curves are known (for example, see \cite{BMS, Cohn}).
The second curve has none, while the only integer points on the first are $(1, \pm 2)$, which is
not of the form $(p,\pm m)$.
\end{proof}

We require the following fundamental divisibility property for Lucas numbers.

\begin{prop}[Prop. 2.1 (ii) of \cite{BHV}]\label{PropA}  If $d\mid n$, then $u_d(\alpha, \beta) | u_n(\alpha,\beta).$
\end{prop}

\subsection{Properties of $\tau(n)$}

Here we record properties enjoyed by Ramanujan's tau-function.
 These include the Hecke multiplicativity established
by Mordell \cite{Mordell}, and 
 the deep theorem of Deligne  \cite{Deligne1, Deligne2} that bounds $|\tau(p)|.$ 

\begin{thm}\label{Newforms} 
The following are true:
\begin{enumerate}
\item If $\gcd(n_1,n_2)=1,$ then $\tau(n_1 n_2)=\tau(n_1)\tau(n_2).$
\item If $p$ is prime and $m\geq 2$, then
$$
\tau(p^m)=\tau(p)\tau(p^{m-1}) -p^{11}\tau(p^{m-2}).
$$
\item If $p$ is prime and $\alpha_p$ and $\beta_p$ are roots of $F_p(X):=X^2-\tau(p)X+p^{11},$ then
$$
   \tau(p^m)=u_{m+1}(\alpha_p,\beta_p)=\frac{\alpha_p^{m+1}-\beta_p^{m+1}}{\alpha_p-\beta_p}.
$$   
Moreover, we have $|\tau(p)|\leq 2p^{\frac{11}{2}}$, and $\alpha_p$ and $\beta_p$ are complex conjugates.
\end{enumerate}
\end{thm}

\subsection{Integer Points on certain hyperelliptic curves and Thue curves}\label{IntegerPoints}
To prove Theorem~\ref{Lehmer135}, we require knowledge of the integer points on certain hyperelliptic curves and Thue equations.
Here we include the information we require in the following two subsections.

\subsubsection{Some hyperelliptic curves}
For $d\geq 2,$ we define the hyperelliptic curves
\begin{equation}
H^{\pm}_{d,\ell}:  \ Y^2 =5 X^{2d}\pm 4\ell \ \ \ \ {\text {\rm and}}\ \ \ \ C^{\pm}_{d,\ell}: \ Y^2=X^{2d-1}\pm \ell.
\end{equation}
The following satisfying lemma classifies the integer points on $H_{d,5}^{\pm}.$
\begin{lem}\label{AnnalsCorollary}
The folllowing are true.
\begin{enumerate}
\item If $d=2$ and $\ell=5$, then the only integer points on $H^{+}_{2,5}$ are $(\pm 1, \pm 5)$ and $(\pm 2, \pm 10)$.
\item If $d>2,$ then the only integer points on $H^{+}_{d,5}$ are $(\pm 1, \pm 5).$
\item If $d\geq 2,$ then $H^{-}_{d,5}$ has no integer points.
\end{enumerate}
\end{lem}
\begin{proof} 
We recall the classical Lucas sequence $$\{L_n\}=\{2, 1, 3, 4, 7, 11, 18, 29, 47,76, 123, 199, 322, 521, 843,\dots\},$$ defined by
$L_0:=2$ and $L_1:=1$ and the recurrence $L_{n+2}:=L_{n+1}+L_n$ for $n\geq 0$.
Bugeaud, Mignotte, and Siksek \cite{AnnalsFibonacci} proved that $L_1=1$ and $L_3=4$
are the only perfect power Lucas numbers.
By the theory of Pell's equations, the positive integer $X$-coordinate solutions to
$$ Y^2 = 5X^2 + 20 \ \ \ \ {\text {\rm and}}\ \ \ \ Y^2=5X^2-20,$$
namely $\{L_1=1, L_3= 4,L_5=11,L_7= 29,\dots\}$ and
$\{L_0=2, L_2=3, L_4=7, L_6=18,\dots\}$ respectively, split the Lucas numbers. The three claims follow immediately.
\end{proof}

The following satisfying lemma classifies the integer points on $H_{11,691}^{+}$ and $C_{6,691}^{+}.$

\begin{lem}\label{Plus691}
There are no integer points on $C^{+}_{6,691}$ and $H^{+}_{11,691}.$
\end{lem}
\begin{proof}
We carry out the Chabauty--Coleman method \cite{coleman} to determine the integral points on these curves. 

The genus 5 curve $C^{+}_{6,691}$ has Jacobian with Mordell-Weil rank 0, as can be found using the implementation of 2-descent in \texttt{Magma} \cite{magma}. Since the rank is less than the genus, we may apply the Chabauty--Coleman method, which, in this case, gives a 5-dimensional space of regular 1-forms vanishing on rational points. We take as our basis for the space of annihilating differentials the set $\{\omega_i := X^i \frac{dX}{2Y}\}_{i = 0, 1, \ldots, 4}.$ The prime $p = 3$ is a prime of good reduction for $C^{+}_{6,691}$, and taking the point at infinity $\infty$ as our basepoint, we compute the set of points $$\left\{z \in C^{+}_{6,691}(\Z_3): \int_{\infty}^z \omega_i = 0\;\textrm{for all}\;{i = 0, 1, \ldots, 4}\right\},$$ where the integrals are Coleman integrals computed\footnote{\texttt{SageMath} code used in this paper can be found in \cite{github}.} using \texttt{SageMath} \cite{sage}. This set, by construction, contains the set of integral points on the working affine model of $C^{+}_{6,691}$. 

The computation gives three points: two points with $X$-coordinate 0 and a third point with $Y$-coordinate 0 in the residue disk of $(2,0)$. (Indeed, the power series corresponding to the expansion of the integral of $\omega_0$ has each of these points occurring as simple zeros.)  We conclude that are no integral points on $C^{+}_{6,691}$.

To compute integral points on $H^{+}_{11,691}$, we reduce to considering integral points on the curve $Y^2 = 5X^{11} + 4 \cdot 691$ and then pull back any points found using the map $(X,Y) \rightarrow (X^2,Y)$. Using \texttt{Magma}, we find that the rank of the Jacobian of this genus 5 curve is 0. We rescale variables to work with the monic model $Y^2 = X^{11}+4\cdot 5^{10}\cdot 691$ and run the Chabauty--Coleman method using $p = 3$. As before, the computation gives three points with coordinates in $\Z_3$: two points with $X$-coordinate 0 and a third point with $Y$-coordinate 0 in the residue disk of $(2,0)$. As before, the power series corresponding to the expansion of the integral of $\omega_0$ has each of these points occurring as simple zeros.  None of these points are rational, and thus we conclude that there are no integral points on $H^{+}_{11,691}$.
\end{proof}

In contrast to the algebraic method used to establish Lemma~\ref{Plus691}, we 
show that there are no integer points on $H_{11,691}^{-}$ and $C_{6,691}^{-}$
using the classical analytic method of Thue equations.\footnote{We could have used the Thue method to provide an alternate proof of Lemma~\ref{Plus691}.} We  use the classical fact that these hyperelliptic equations can be reduced to the setting of Thue equations. A {\it Thue equation} is an equation of the form
$$F(X,Y)=m,
$$
where $F(X,Y)\in \Z[X,Y]$ is homogeneous and $m$ is a non-zero integer.
Thanks to work of Bilu and Hanrot \cite{BH96}, many of these equations can be effectively
solved using software packages such as \texttt{PARI/GP} \cite{pari} and \texttt{Magma}.

\begin{lem}\label{Minus691}
There are no integer points on $C^{-}_{6,691}$ and $H^{-}_{11,691}.$
\end{lem}
\begin{proof}
Generalized Lebesgue--Ramanujan--Nagell equations are Diophantine equations of the form
\begin{equation}\label{GLRN}
x^2+D=Cy^n,
\end{equation}
where $D$ and $C$ are non-zero integers. An integer point on (\ref{GLRN}) can be studied in the ring of integers
of $\Q(\sqrt{-D})$
using the factorization
$$(x+\sqrt{-D})(x-\sqrt{-D}) =Cy^n.
$$

This observation is a standard tool in the study of Thue equations. In particular, 
Theorem~2.1 of  \cite{Barros} (also see Proposition~3.1 of \cite{BMS}) gives a step-by-step algorithm that takes alleged solutions of (\ref{GLRN}) and produces integer points on one
of finitely many Thue equations constructed from $C, D$ and $n$ via the algebraic number theory of $\Q(\sqrt{-D})$.
 These equations are assembled from the knowledge of the group of units and the
ideal class group.

The  hyperelliptic curve $C^{-}_{6,691}$ corresponds to (\ref{GLRN}) for the class number 5 imaginary quadratic field
$\Q(\sqrt{-691})$, where $x=Y, y=X, C=1, D=691,$ and $n=11.$  In this case the algorithm gives exactly one Thue equation, which after clearing denominators, can be rewritten as
\small
\begin{displaymath}
\begin{split}
2\times5^{55}&=(991077174272090396)x^{11} + (119700018439220789119)x^{10}y
- (8831599221002836172345)x^9y^2\\
&\ \ \ \ -(337116345512786456280840)x^8y^3 
+ (8492967300375371034332430)x^7y^4\\
&\ \ \ \  + (175189311986919278870504298)x^6y^5 
- (1881807368163995585644810248)x^5y^6\\
&\ \ \ \  - (22992541672786450593030038430)x^4y^7 
+ (104772541553739359102253613965)x^3y^8\\
&\ \ \ \  + (697875798749922445133117312720)x^2y^9 
- (1068801486169809452619368218519)xy^{10}\\
&\ \ \ \  - (2292300374810647823111384294421)y^{11}.
\end{split}
\end{displaymath}
\normalsize
The Thue equation solver in \texttt{PARI/GP}, which implements the Bilu--Hanrot algorithm, establishes that there are no integer solutions, and so $C^{-}_{6,691}$ has no integer points.

We now turn our attention to the hyperelliptic curve $H^{-}_{11,691}.$  Its integer points $(X,Y)$  satisfy
$$
(Y+2\sqrt{-691})(Y-2\sqrt{-691})=5X^{22}.
$$
Therefore, we again employ the imaginary quadratic field $\Q(\sqrt{-691})$. In particular, we have (\ref{GLRN}), where $x=Y, y=X, C=5, D=4\cdot 691$ and $n=22$.
 The algorithm again gives one Thue equation, which after clearing denominators, can be rewritten as
\small 
\begin{displaymath}
\begin{split}
2^2\times5^{110}&=
-(20587212586465949627980680671826599752) x^{22} \\
&\ \  \ \   + (1133274396835827658613802749227310922394) x^{21} y\\ 
&\ \ \ \ +\cdots\\
%+ (834626157759714935398436870555093170574313) x^{20} y^{2}
 %-(8159778560617106709823572\\156108780799695920) x^{19} y^{3}
  %-(4611119249089902911296098623599168794928167005) x^{18} y^{4}
%+ (7539043619\\149936299005118464459598577999038918) x^{17} y^{5}
%+(8158141650531634163453610133519062118167907440\\624) x^{16} y^{6}
%+(10200580045436130409656231389386168872804674737488) x^{15} y^{7}  
%-(602955322188040458524\\2499793711176989547982608701290) x^{14} y^{8}
%-(145263476930736101573607064343340463608106650748261\\60) x^{13} y^{9} 
 %+ (2090521445155964515361490509466367058541032584486823958) x^{12} y^{10} 
 %+ (5844542226903512\\116092996421688829911004383083434793464) x^{11} y^{11} -(356302712970653641717785944751133407042511\\285956405317392) x^{10} y^{12} -(987038231631104311380666892787656679817402335849568959660) x^{9} y^{13}
 %+ (2\\9846976746440868116013127106148303934688267302455116612710) x^{8} y^{14} 
 %+ (744638552799924406264336\\05011674579259479283015320019153088) x^{7} y^{15}  -(11722450246463330168423318619922455737375608704\\56569032220251) x^{6} y^{16} -(2402904025112033247090185414581187749919787918945389287718682) x^{5} y^{17}  +\\(19214721685555377553624302625818148918152807008363513685513120) x^{4} y^{18}  +(2835528622294626379\\7640325507338644259735287439044502692583580) x^{3} y^{19} -(100719765538171123065417246571789932713\\895568329834520414863087) x^{2} y^{20}
 &\ \ \ \  -(79670423145107301772779399379735976309907264511718034789276856) x y^{21}\\
&\ \ \ \  +(71809437208138431262783549625248617351731199323326115439324273)y^{22}.
 \end{split}
 \end{displaymath}
\normalsize
The Thue solver in \texttt{SageMath} establishes that there are no integer solutions, and so $H^{-}_{11,691}$ has no integer points.

\end{proof}

\subsubsection{Some Thue equations}
We require Thue equations that arise
from the generating function
\begin{equation}\label{genfunction}
\frac{1}{1-\sqrt{Y}T+XT^2}=\sum_{m=0}^{\infty}F_m(X,Y)\cdot T^{m}=1+\sqrt{Y}\cdot T+(Y-X)T^2+\cdots.
\end{equation}
For every positive integer $m$, it is simple to verify that
\begin{equation}\label{ThueEqn}
F_{2m}(X,Y)=\prod_{k=1}^m \left(Y-4X \cos^2\left(\frac{\pi k}{2m+1}\right)\right).
\end{equation}
The first few homogenous polynomials $F_{2m}(X,Y)$ are as follows:
\begin{displaymath}
\begin{split}
F_2(X,Y)&=Y-X,\\
F_4(X,Y)&=Y^2-3XY+X^2\\
F_6(X,Y)&=Y^3-5XY^2+6X^2Y-X^3.\\
\end{split}
\end{displaymath}

We require the following lemma about six Thue equations arising from these polynomials.
\begin{lem}\label{Monster} The following are true.
\begin{enumerate}
\item The points $(\pm 1, \pm 4), (\pm 2, \pm 1), (\mp3, \mp 5)$ are the only integer solutions to
$$
F_6(X,Y)=\pm 7.
$$
\item There are no integer solutions to
$$
F_{22}(X,Y)=\pm 691.
$$
\item The points $(\pm 1, \pm 4)$ are the only integer solutions to
$$
F_{690}(X,Y)=\pm 691.
$$
\end{enumerate}
\end{lem}

\begin{proof} Claims (1) and (2) are easily obtained using the Thue solver in \texttt{PARI/GP}.

At first glance, the proof of (3) seems far more formidable, as $F_{690}(X,Y)$ is a degree 345 homogeneous polynomial.
However, for odd primes $p$, the Thue equations $F_{p-1}(X,Y)=\pm p$
 are essentially the well-studied equations
\begin{equation}\label{ModifiedThue}
\widehat{F}_p(X,Y)=\prod_{k=1}^{\frac{p-1}{2}}\left(Y-2X\cos\left(\frac{2\pi k}{p}\right)\right)=\pm p
\end{equation}
that starred in the work of Bilu, Hanrot, and Voutier on primitive prime divisors of Lucas sequences. 
Indeed, we have that $F_{p-1}(X,Y)=\widehat{F}_p(X,Y-2X).$
A key step 
(see Cor. 6.6 of \cite{BHV}) in their work is that  there are no integer solutions to (\ref{ModifiedThue})
with $|X|>e^8$  when $31\leq p\leq 787.$ By a standard lemma (for example, see Lemma~1.1 of \cite{TW} and Proposition 2.2.1 of \cite{BH96})), midsize solutions of $\widehat{F}_{691}(X,Y)=\pm 691$  correspond to convergents of the continued fraction expansion of some
$2\cos(2\pi k/691).$  A simple calculation rules out this possibility, leaving only potential small solutions, those  with
 $|X|\leq 4$.  For these $X$ we find the solutions $(\pm 1, \pm 2)$, which implies that
$(\pm 1, \pm 4)$ are  indeed the only integral solutions to $F_{690}(X,Y)=\pm 691.$
\end{proof}

\section{Proof of Theorem~\ref{Lehmer135}}
It is well-known that $\tau(n)$ is odd if and only if $n$ is an odd square. 
To see this, we employ 
 the Jacobi Triple Product identity to obtain the congruence
\begin{displaymath}
\begin{split}
\sum_{n=1}^{\infty}\tau(n)q^n
&:=q\prod_{n=1}^{\infty}(1-q^n)^{24}
\equiv q\prod_{n=1}^{\infty}(1-q^{8n})^3\pmod 2\\
&\;=\sum_{k=0}^{\infty} (-1)^k(2k+1)q^{(2k+1)^2}.
\end{split}
\end{displaymath}

We consider the possibility that $\pm 1$ appears in sequences of the form
\begin{equation}\label{primepowers}
\{\tau(p),\tau(p^2), \tau(p^3),\dots\}.
\end{equation}
By Theorem~\ref{Newforms} (2), if  $p\mid \tau(p)$ is prime, then $p^m\mid \tau(p^m)$ for every $m\geq 1$, and
so $|\tau(p^m)|\neq 1.$  
For primes
$p\nmid \tau(p),$ Theorem~\ref{Newforms} (3) gives a Lucas sequence 
satisfying Lemma~\ref{Modularity}, which in turn implies that there are no defective terms with
$u_{m+1}(\alpha_p,\beta_p)=\tau(p^m)=\pm 1$. Therefore,
all of the values in (\ref{primepowers}) always have a prime divisor, and so cannot have absolute value 1.

We now turn to the primality of absolute values of $\tau(n)$.
Thanks to Hecke multiplicativity (i.e. Theorem~\ref{Newforms} (1)) and the discussion above,  if $\ell$ is an odd prime and $|\tau(n)|=\ell$, then $n=p^d$, where $p$ is an odd prime for which
 $p\nmid \tau(p).$ The fact that $\tau(p^d)=u_{d+1}(\alpha_p,\beta_p)$ leads to a further constraint on $d$ 
(i.e. refining 
 the fact that $d$ is even).
By Proposition~\ref{PropA}, which guarantees relative divisibility between Lucas numbers, and 
 Lemma~\ref{Modularity}, which guarantees the absence of defective terms in (\ref{primepowers}), it follows that 
$d+1$ must be an odd prime, and $\tau(p^d)$ is the very first term that is divisble by $\ell$.
To make use of this observation, for odd primes 
$p$ and $\ell$ we define
\begin{equation}
m_{\ell}(p):=\min\{ n\geq 1\ : \ \tau(p^n)\equiv 0\!\!\!\!\pmod{\ell}\}.
\end{equation}
For $|\tau(p^d)|=\ell$, we must have $m_{\ell}(p)=d,$ where $d+1$ is also an odd prime.

Thanks to the mod 3 congruence in (\ref{RamanujanCongruences}), we find that
$$
m_3(p)=\begin{cases} 1 \ \ \ \ \ &{\text {\rm if }} p=0, 2\!\!\!\!\pmod 3,\\
2\ \ \ \ \ &{\text {\rm if }} p\equiv 1\!\!\!\!\pmod 3.
\end{cases}
$$
Therefore, $d=2$ is the only possibility.
By Theorem~\ref{Newforms} (2), if $\tau(p^2)=\pm 3$, then
 $(p,\tau(p))\in C^{\pm}_{6,3}(\Z)$.
 However, recall that in (\ref{hyper1}) we used the fact that there are no such integer points.
 
Thanks to the mod 5 congruence in (\ref{RamanujanCongruences}), we find that
$$
m_5(p)=\begin{cases} 1 \ \ \ \ \ &{\text {\rm if }} p\equiv 0, 4\!\!\!\!\pmod{5}\\
3 \ \ \ \ \ &{\text {\rm if }} p\equiv 2, 3\!\!\!\!\pmod{5}\\
4 \ \ \ \ \ &{\text {\rm if }} p\equiv 1\!\!\!\!\pmod{5},
\end{cases}
$$
and so we only need to consider $d=4.$
By Theorem~\ref{Newforms} (2), if $\tau(p^4)=\pm 5$, then
 $(p, 2\tau(p)^2-3p^{11})\in H^{\pm}_{11,5}(\Z)$.
Lemma~\ref{AnnalsCorollary} (2) and (3) show that no such points exist.

Thanks to the mod 7 congruence in (\ref{RamanujanCongruences}), we find that
$$
m_7(p)=\begin{cases}
1 \ \ \ \ \ &{\text {\rm if }} p\equiv 0, 3, 5, 6\!\!\!\!\pmod{7}\\
6 \ \ \ \ \ &{\text {\rm if }} p\equiv 1, 2, 4\!\!\!\!\pmod{7}.
\end{cases}
$$
Hence, $d=6$ is the only possibility, and so we must rule out the possibility that $\tau(p^6)=\pm 7$.
Thanks to (\ref{genfunction}) and Theorem~\ref{Newforms} (3), 
for every $m\geq 1$ we have
$$
F_{2m}(p^{11},\tau(p)^2)=\tau(p^{2m}).
$$
Lemma~\ref{Monster} (1) shows that there are no such solutions to $F_6(X,Y)=\pm 7.$

Thanks to the mod 691 congruence in (\ref{RamanujanCongruences}), we find that
the only cases where $m_{691}(p)=d$, where $d+1$ is an odd prime, are $d=2, 4, 22,$ and $690$.
By Lemma~\ref{Monster} (2) and (3), the latter two cases, which correspond to
$$
\tau(p^{22})=F_{22}(p^{11},\tau(p)^2)=\pm 691 \ \ \ {\text {\rm and}}\ \ \ 
\tau(p^{690})=F_{690}(p^{11},\tau(p)^2)=\pm 691,
$$
have no such solutions.  If $\tau(p^2)=\pm 691,$ then $(p,\tau(p))\in C^{\pm}_{6,691}(\Z)$.
If $\tau(p^4)=\pm 691$, then $(p, 2\tau(p)^2-3p^{11})\in H^{\pm}_{11,691}(\Z)$.
Lemmas~\ref{Plus691} and \ref{Minus691} show that no such integer points exist.
\qed

\end{document}